\documentclass{amsart}
\usepackage[utf8]{inputenc}

\newtheorem{thm}{Theorem}[section] 

\newtheorem{lem}[thm]{Lemma}

\newtheorem{rem}[thm]{Remark}

\newtheorem{ques}[thm]{Question}

\begin{document}

\title{Generating Adjoint Groups}
\author{Be'eri Greenfeld}
\thanks{The author thanks Prof.~Agata Smoktunowicz for helpful discussions.}
\maketitle
\begin{abstract}
    We prove two approximations of the open problem of whether the adjoint group of a non-nilpotent nil ring can be finitely generated: We show that the adjoint group of a non-nilpotent Jacobson radical cannot be boundedly generated, and on the other hand construct a finitely generated, infinite dimensional nil algebra whose adjoint group is generated by elements of bounded torsion.
\end{abstract}

\section{Introduction}

Nil rings give rise to multiplicative groups, called adjoint groups. If $R$ is nil (or, more generally Jacobson radical) then its adjoint group $R^\circ$ consists of the elements of $R$ as underlying set, with multiplication given by $r\circ s = r+s+rs$. We can think of $R^\circ$ as $\{1+r|r\in R\}$ with usual multiplication $(1+r)(1+s)=1+r+s+rs$. The connections between group-theoretic properties of $R^\circ$ and ring-theoretic properties of $R$ were intensively studied (for instance, see \cite{sys1, sys2, sys3, Engel}).

In this way, nil rings give rise to braces (in fact, two-sided braces, namely braces arising from Jacobson radical rings). For more information about braces and their connections with nil rings see \cite{Rump, Smoktunowicz, Engel}.

The following is an open question posed by Amber, Kazarin, Sysak \cite{3, 4} and then repeated by Smoktunowicz \cite{Smoktunowicz}:

\begin{ques} \label{question}
Is there a non-nilpotent, nil algebra whose adjoint group is finitely generated?
\end{ques}

Note that by \cite[~Theorem 4.3]{sys4}, if $R$ is Jacobson radical such that $R^\circ$ is generated by two elements then $R$ is nilpotent.

Bounded generation is a stronger notion of finite generation: it asserts that a group is generated by a finite set $\{g_1,\dots,g_n\}$ such that every element can be expressed as $g_1^{i_1}\cdots g_n^{i_n}$, namely, all elements have finite width with respect to the generating set. Boundedly generated groups are extensively studied, and for related information and results we refer the reader to \cite{bdd3, bdd5, bdd9, bdd33, pig4, bdd43}.

It is easy to see that the adjoint group of a non-nilpotent nil algebra cannot be boundedly generated; in Theorem \ref{bdd} we prove that the adjoint of a non-nilpotent Jacobson radical cannot be boundedly generated.

Suppose $R$ is a nil algebra of positive characteristic $p>0$, then its adjoint group $R^{\circ}$ is residually-p torsion. However, its torsion cannot be bounded, since (unless the algebra is nilpotent) its elements' nilpotency indices are unbounded.

In Theorem \ref{exmpl} we construct a finitely generated, infinite dimensional nil ring (in fact, a Golod-Shafarevich algebra) whose adjoint group is generated by elements of bounded torsion. Indeed, if the answer to Question \ref{question} is affirmative then it evidently yields such example; in this sense, Theorem \ref{exmpl} should be seen as an approximation of Question \ref{question}.

Our example also results in an infinite, finitely generated brace whose adjoint group is generated by a set of elements of bounded torsion.
Note that the analogous question of Question \ref{question} for braces has a positive solution: for any finite, non-degenerate involutive solution $(X,r)$ of the Yang-Baxter equation there is an associated brace $G(X,r)$, which is infinite, both of whose multiplicative and additive groups are finitely generated and also its adjoint group is; note that the additive group of $G(X,r)$ is a free abelian group (our two-sided braces have p-torsion additive groups). For details, see \cite[~Theorem 4.4]{Cedo}.

\section{Boundedly generated adjoint groups}
Recall that a group is \textit{boundedly generated} if $G=H_1\cdots H_n$ for some cyclic subgroups $H_1,\dots,H_n\subseteq G$. The minimal such $n$ is called the \textit{cyclic width} of $G$.

Note that the adjoint group of an infinite dimensional nil algebra cannot be boundedly generated, since that would mean the algebra admits a Shirshov base, namely every element would be a linear combination of monomials of the form $f_1^{i_1}\cdots f_n^{i_n}$ for some $f_1,\dots,f_n$; since the algebra is nil, there are only finitely many such monomials, so the algebra must be finite dimensional.

We now extend this result to an arbitrary Jacobson radical. Note that Jacobson radicals need not be nil, so the above argument does not hold any longer.

\begin{thm} \label{bdd}
Let $R$ be an algebra which is Jacobson radical. Then $R^\circ$ is boundedly generated if and only if $R$ is finite dimensional and nilpotent.
\end{thm}

\begin{proof}
First observe that we only have to deal with the case where the base field $F$ is finite. Indeed, observe by Nakayama's lemma that $R/R^2$ is a non-zero $F$-vector space. Therefore the additive group of $R/R^2$, being isomorphic to $(R/R^2)^\circ$ occurs as a homomorphic image of $R^\circ$. Hence, if $R^\circ$ is finitely generated then so is the additive group of the base field $F$, from which it follows that $F$ must be finite. We henceforth assume $F=\mathbb{F}_p$.

Assume that $R$ is an algebra which is Jacobson radical and $R^\circ$ is boundedly generated with cyclic width $m$. Denote $G_n=(R^{n+1})^\circ\triangleleft R^\circ$. Note that $G_n$ has finite index since $R/R^{n+1}$ is finite $F$-dimensional.
Given a finite group $G$ denote by $exp(G)$ the minimum power $\alpha$ such that $g^\alpha=e$ for all $g\in G$.

Observe that $exp(R^\circ/G_n)\leq p(n+1)$ since given $f\in R$, we have that: $$(1+f)^{p^{\lceil \log_p (n+1) \rceil}}=1+f^{p^{\lceil \log_p (n+1)\rceil}}\in G_n$$
(as: $f^{p^{\lceil \log_p (n+1)\rceil}}\in R^{n+1}$ and $n+1\leq p^{\lceil \log_p (n+1)\rceil}\leq p(n+1)$.)

Note that if a group $G$ is boundedly generated of cyclic width $m$ then for any finite index normal subgroup $H$ we have that $[G:H]\leq exp(G/H)^m$ and in particular in our case: $$p^{\dim_{\mathbb{F}_p}(R/R^{n+1})}=[R^\circ:G_n]\leq exp(R^\circ/G_n)^m\leq (p(n+1))^m.$$

It follows that (denoting $R^0=R$): $$\sum_{i=0}^{n}\dim_{\mathbb{F}_p}(R^{i}/R^{i+1})=\dim_{\mathbb{F}_p}(R/R^{n+1})\leq m(1+\log_p (n+1))<n$$
for $n\gg 1$, hence for some $0\leq i\leq n$ we have that $R\cdot R^i=R^{i+1}=R^i$. Since $R$ is a finitely generated algebra which is Jacobson radical it follows by Nakayama's lemma that $R^i=0$. Hence $R$ is a finite dimensional, nilpotent algebra.
\end{proof}

\begin{rem}
Note that in fact, Theorem \ref{bdd} requires a weaker assumption on $R^\circ$ than being boundedly generated: if $R$ is a Jacobson radical whose adjoint group is a \textit{polynomial index growth (PIG)} group, namely group in which the index of any subgroup is polynomially bounded by its relative exponent, then $R$ is nilpotent. For more information and results on PIG groups, see \cite{pig1, pig2, pig3, pig4}.
\end{rem}

\section{Adjoint group generated by a uniformly torsion set}

Let $F$ be a countable field of characteristic $p>0$. Consider the free algebra $A=F\left<x,y\right>$. Observe that $A=\bigoplus_{n=0}^{\infty} A(n)$ is graded by $\deg(x)=\deg(y)=1$. Denote $A^{+}=\bigoplus_{n=1}^{\infty}A(n)$. For $a\in A$, let $v(a)=\max\{d|a\in \bigoplus_{i\geq d}A(i)\}$.

\begin{lem} \label{far_enough}
Let $a\in A^{+}$ and write $a=a_1+\cdots+a_n$ a sum of homogeneous elements. For every $m\geq 1$ there exist some homogeneous elements $h_1,\dots,h_k$ and an element $b\in A^{+}$ with $v(b)\geq m$ such that: $$1+a+b=(1+h_1)\cdots(1+h_k).$$
\end{lem}

\begin{proof}
Assume on the contrary that $b\in A^{+}$ is such that $1+a+b=(1+h_1)\cdots(1+h_k)$ for some homogeneous $h_1,\dots,h_k$ with $d=v(b)$ maximal. Write $b=b_1+\cdots+b_m$ a sum of homogeneous elements. By maximality of $d$ we get that $d=\min\{\deg(b_i)|1\leq i\leq m\}$.

Note that:
\begin{eqnarray*}
(1+h_1)\cdots(1+h_k)(1-b_1)\cdots(1-b_m) & = & (1+a+b)(1-b_1)\cdots(1-b_m)= \\ (1+a+b)(1-b+c)&=&1+a+(c+ac+bc-ab-b^2).
\end{eqnarray*}

with: $$c=\sum_{S\subseteq \{1,\dots,m\};|S|>1} \prod_{i\in S}b_i.$$ Observe that $v(ab),v(b^2),v(c)>d$ and hence $v(c+ac+bc-ab-b^2)>d$, contradicting the maximality of $d$.
\end{proof}

\begin{thm} \label{exmpl}
There exists a finitely generated, graded, infinite dimensional nil algebra $R$ whose adjoint group $R^{\circ}$ is generated by a set of elements of bounded torsion.
\end{thm}

\begin{proof}

Enumerate $A^+=\{f_1,f_2,\dots\}$. Using Lemma \ref{far_enough} pick $b_1$ with $v(b_1)\geq 14$ such that $1+f_1+b_1=(1+h_{1,1})\cdots(1+h_{1,k_1})$ for some homogeneous elements $h_{1,1},\dots,h_{1,k_1}$. Write $b_1=b_{1,1}+\cdots+b_{1,n_1}$ a sum of homogeneous elements with $\deg(b_{1,d})=d$. Next use Lemma \ref{far_enough} to choose $b_2$ with $v(b_2)\geq n_1+1$ and $1+f_2+b_2=(1+h_{2,1})\cdots(1+h_{2,k_2})$ for some homogeneous elements $h_{2,1},\dots,h_{2,k_2}$. Again, write $b_2=b_{2,n_1+1}+\cdots+b_{2,n_2}$ a sum of homogeneous elements with $\deg(b_{1,d})=d$, and inductively proceed, each time constructing $b_{l+1}$ with $v(b_{l+1})\geq n_l+1$ such that $1+f_{l+1}+b_{l+1}=(1+h_{l+1,1})\cdots(1+h_{l+1,k_{l+1}})$ for suitable homogeneous elements $h_{l+1,1},\dots,h_{l+1,k_{l+1}}$.

Let $I\triangleleft A^+$ be the ideal generated by $\{b_{i,j}|i\geq 1\}$. Note that $I$ is generated by at most one homogeneous element of each degree, starting from $14$.

Now let $J$ be the ideal generated by $h^{p^\alpha}$ (with $\alpha=\lceil \log_p 7\rceil$) for all homogeneous elements $h\in A^+$. Observe that $J$ is generated by $2^d$ elements of degree $p^\alpha d\geq 7d$ for $d\geq 1$.

We claim that $A^+/(I+J)$ is a Golod-Shafarevich algebra. Indeed, let $r_n$ denote the number of degree $n$ generators of $I+J$. Then for all $\tau>0$ for which the series are convergent: $$\sum_{n>1} r_n\tau^n\leq\sum_{d=1}^{\infty}2^d\tau^{7d}+\sum_{n=3} \tau^n=\frac{2\tau^7}{1-2\tau^7}+\frac{\tau^{14}}{1-\tau}.$$

Setting $f(\tau)=1-2\tau+\sum_{n>1} r_n\tau^n$, we get that: $$f(0.75)\approx-0.5+0.3642+0.071<0,$$
hence $A^+/(I+J)$ is a Golod-Shafarevich algebra. Note that the semigroup $\left(A^+/(I+J)\right)^\circ$ is generated by the set $\{1+f|f\ homogeneous\}$, whose elements are torsion of order $\leq p^\alpha$. Now by \cite{Wilson} it follows that $A^+/(I+J)$ admits a homomorphic image which is an infinite dimensional and nil algebra, say $R$. Then $R$ satisfies the claimed properties.

\end{proof}

\section{Further questions}

We conclude with two open questions related to Question \ref{question}:

\begin{ques}
Is there an infinite dimensional, finitely generated nil algebra with polynomial growth whose adjoint group is finitely generated?
\end{ques}

In \cite{amb1} it is proved that every elementary amenable subgroup of the adjoint group of a nil ring is locally nilpotent.

\begin{ques}
Is there an infinite dimensional, finitely generated nil algebra $R$ whose adjoint group $R^\circ$ is amenable? Is it possible that $R^\circ$ has intermediate growth?
\end{ques}

\end{document}